\documentclass[12 pt]{article}%
\usepackage{amsmath, amsfonts, amsthm, color,latexsym}
\usepackage{amsmath, ulem}
\usepackage{amsfonts}
\usepackage{amssymb}
\usepackage{color, soul}
\usepackage[all]{xy}
\usepackage{graphicx}%
\providecommand{\U}[1]{\protect\rule{.1in}{.1in}}
\allowdisplaybreaks[4]
\newtheorem{theorem}{Theorem}[section]
\newtheorem{proposition}[theorem]{Proposition}
\newtheorem{corollary}[theorem]{Corollary}

\newtheorem{remark}[theorem]{Remark}

\newtheorem{lemma}[theorem]{Lemma}
\newtheorem{final remark}[theorem]{Final Remark}
\newtheorem{definition}[theorem]{Definition}
\allowdisplaybreaks[4]

\begin{document}

\title{Zero sets of homogeneous polynomials containing infinite dimensional spaces }
\author{Mikaela Aires\thanks{Supported by a CNPq scholarship}~~and Geraldo Botelho\thanks{Supported by FAPEMIG grants RED-00133-21 and APQ-01853-23. \newline 2020 Mathematics Subject Classification: 15A03,  46B87, 46G25, 47H60.\newline Keywords: Infinite dimensional linear spaces, homogeneous polynomials, finitely lineable sets. }}
\date{}
\maketitle

\begin{abstract} Let $X$ be a (real or complex) infinite dimensional linear space. We establish conditions on a homogeneous polynomial $P$ on $X$ so that, if $W$ is any finite dimensional subspace of $X$ on which $P$ vanishes, then $P$ vanishes on  an infinite dimensional subspace of $X$ containing $W$. In the complex case, this is a step beyond the classical result due to Plichko and Zagorodnyuk. Applications to the real case are also provided.
\end{abstract}

\section{Introduction and main result}

In 1998, Plichko and Zagorodnyuk \cite{plichko} proved the following remarkable result: For any infinite dimensional complex linear space $X$, every $\mathbb{C}$-valued homogeneous polynomial on $X$ vanishes on an infinite dimensional subspace of $X$. The real case of the problem, which is clearly very different from the complex case, was thoroughly studied by several authors, see, e.g., \cite{realstory, aronhajek, moslehian}.

 All homogeneous polynomials considered here are scalar-valued. In the modern language of lineability (see \cite{Aron}), the Plichko-Zagorodnyuk theorem asserts that the zero set of any homogeneous polynomial on an infinite dimensional complex space is lineable, meaning that it  contains an infinite dimensional linear space. The subject of lineability is growing quickly, several new interesting approaches to the area have appeared. In this paper, we investigate the problem solved by Plichko-Zagorodnyuk in the complex case under the perspective of the notions introduced in \cite{Favaro} and developed in, e.g., \cite{mikaela, nacib, viniciuspams}. More precisely, we are interested in the following question:

\medskip

{\it Given a homogeneous polynomial $P$ on a (real or complex) infinite dimensional linear space $X$ and given a finite dimensional subspace $W$ of $X$ on which $P$ vanishes, does $P$ vanish on an infinite dimensional subspace of $X$ containing $W$?}
\medskip

Considering, as usual, the $0$-dimensional space $\{0\}$ as a finite dimensional subspace of $X$ and using that $P(0) = 0$, for the answer to the question above to be affirmative, $P$ should vanish on an infinite dimensional subspace of $X$, that is, the zero set of $P$ should be lineable.

In Theorem \ref{main} we establish conditions on $P$ under which  the answer to the question above is affirmative. Applications to the complex case include an extension of the Plichko-Zagorodnyuk theorem (cf. Corollary \ref{ghnv}). Applications of the main result to the real case are also given, see, e.g., Proposition \ref{umd3} and Remark \ref{ytnb}.

The following definition is given just for the sake of simplicity.

\begin{definition}\rm A nonempty subset $A$ of an infinite dimensional linear space $X$ is {\it finitely lineable} if, for every finite dimensional subspace $W$ of $X$ contained in $A \cup \{0\}$, there exists an infinite dimensional subspace of $X$ containing $W$ and contained in $A \cup \{0\}$.
\end{definition}

Note that every finitely lineable set contains, up to $0$, an infinite dimensional space, that is, it is lineable. Since $P(0) = 0$ for every homogeneous polynomial $P$, the answer to the question above is affirmative if and only if the zero set of $P$ is finitely lineable if and only if (in the language of \cite{Favaro}) the zero set of $P$ is $(n,\aleph_0)$-lineable for every $n \in \mathbb{N}\cup \{0\}$.

Let $X$ be an infinite dimensional linear space $\mathbb{K} = \mathbb{C}$  or $\mathbb{R}$. Given an $m$-homogeneous polynomial $P \colon X \longrightarrow \mathbb{K}$, by ${\check P} \colon X^m \longrightarrow \mathbb{K}$ we denote the (unique) symmetric $m$-linear form associated to $P$, that is, $P(x) = {\check P}(\underbrace{x, \ldots ,x}_{m\, {\rm times}})$ for every $x \in X$. Given $0 \leq k \leq m$, $x_1, \ldots, x_k \in X$, and $\alpha_1, \ldots, \alpha_k \in \{0, 1, \ldots, k\}$ with $\alpha_1 + \cdots + \alpha_k = m$, we shall use the simplified notation
$${\check P}(x_1^{\alpha_1}, \ldots, x_k^{\alpha_k}) := {\check P}(\underbrace{x_1, \ldots ,x_1}_{\alpha_1\, {\rm times}}, \ldots,\underbrace{x_k, \ldots ,x_k}_{\alpha_k\, {\rm times}}). $$

 For the basics of homogeneous polynomials on linear spaces the reader is referred to \cite{dineen, mujica}.

 We are assuming the Axiom of Choice, therefore every subspace of a linear space is (algebraically) complemented.

From now on, unless stated explictly otherwise, $m \geq 2$ is a natural number and $X$ is an infinite dimensional linear space over $\mathbb{K} = \mathbb{C}$  or $\mathbb{R}$. Of course, the zero set of a function $f \colon X \longrightarrow \mathbb{K}$ is defined by $\{x \in X : f(x) = 0\}$. Sometimes the zero set of $f$ will be denoted by $f^{-1}(0)$.

\begin{definition} \rm Let $P \colon X \longrightarrow \mathbb{K}$ be an $m$-homogeneous polynomial. For a $t$-homogeneous polynomial $Q \colon X \longrightarrow \mathbb{K}$, $1 \leq t \leq m-1$, we write $Q \prec P$ if there are $x_1,\dots, x_n \in X$ on which $P$ vanishes, and $\alpha_1, \dots, \alpha_n \in \mathbb{N}\cup\{0\}$ with $\alpha_1+ \cdots + \alpha_n+t=m$ such that
%Dado $m \in \mathbb{N}$. Sejam $X$ um espaço vetorial, $P \in \mathcal{P}(^mX)$ e $Y$ um subespaço de $X$. Para cada  $1 \leq t \leq m-1$, denotaremos por $Q \prec P$ o polinômio $t$-homogêneo $Q: Y \to \mathbb{K},$ definido por
$$Q(x)={\check P}(x_1^{\alpha_1}, \dots , x_n^{\alpha_n}, x^t) \mbox{ for every } x \in X.$$
%onde $x_1,\dots, x_n \in X$, $\alpha_1, \dots, \alpha_n \in \mathbb{N}_0$ e $\alpha_1+ \cdots + \alpha_n+t=m$.
\end{definition}

Our main result reads as follows.

\begin{theorem}\label{main} Let $P \colon X \longrightarrow \mathbb{K}$ be an $m$-homogeneous polynomial. Suppose that, for every infinite dimensional subspace $Y$ of $X$, $P$ vanishes on some nonzero vector of $Y$ and every homogeneous polynomial $Q \prec P$ vanishes on an infinite dimensional subspace of $Y$. Then the zero set of $P$ is finitely lineable.
\end{theorem}

The theorem shall be proved in the next section and consequences shall be derived in the last section.

\section{Proof of the main result}

\begin{lemma}\label{3} Let $k \in \mathbb{N}$ and let $f_1, \ldots, f_k \colon X \longrightarrow \mathbb{K}$ be maps such that, regardless of the infinite dimensional subspace $Y$ of $X$, each $f_j$ vanishes on some infinite dimensional subspace of $Y$. Then, all of them vanish on a same infinite dimensional subspace of $X$.
\end{lemma}

\begin{proof} By assumption, there is an infinite dimensional subspace $X_1$ of $X$ on which $f_1$ vanishes. By assumption again, there is an infinite dimensional subspace $X_2$ of $X_1$ on which $f_2$ vanishes. Of course, $f_1$ vanishes on $X_2$ too. Repeating the procedure finitely many times we get the result.
\end{proof}

\noindent{\it Proof of Theorem \ref{main}.} Let $W$ be a finite dimensional subspace of $X$ contained in the zero set of $P$. If $W = \{0\}$, then we set $Y = Y_1 = X$ and we choose $0 \neq y_1 \in Y_1$ such  that $P(y_1) = 0$, which exists by assumption. If $W$ is $n$-dimensional for some $n \in \mathbb{N}$, then we fix a basis  $\{x_1, \ldots, x_n\}$ for $W$. Denote by $Y$ an algebraic complement of  $W$, that is, $X = W \oplus Y$. Of course, $Y$ is infinite dimensional. For $ 1 \leq t \leq m-1$ and $1 \leq i_1, \dots , i_{m-t} \leq n,$  consider the $t$-homogeneous polynomial
$$P^{t,i_1, \dots , i_{m-t}} \colon X \longrightarrow \mathbb{K}~,~P^{t,i_1, \dots , i_{m-t}} (x) = {\check P}(x_{i_1}, \dots x_{i_{m-t}}, x^t).$$
Each $P^{t,i_1, \dots , i_{m-t}} \prec P$, so, by assumption, each $P^{t,i_1, \dots , i_{m-t}}$ vanishes on an infinite dimensional subspace of $Y$. Since there are only finitely many of such polynomials and $Y$ is infinite dimensional, by Lemma \ref{3} there is an infinite dimensional subspace $Y_1$ of $Y$ such that, for all $ 1 \leq t \leq m-1$ and $1 \leq i_1, \dots , i_{m-t} \leq n,$
\begin{equation}\label{ps1}P^{t,i_1, \dots , i_{m-t}}(x)= 0 \mbox{ for every } x\in Y_1.
\end{equation}
Since $Y_1$ is an infinite dimensional subspace of $X$, by assumption there is $0 \neq y_1 \in Y_1$ such that $P(y_1) = 0$. In both cases, $n = 0$ and $n > 0$, $Y$ is an infinite dimensional subspace of $X$,  $X = W \oplus Y$, $Y_1$ is an infinite dimensional subspace of $Y$ and $0 \neq y_1 \in Y_1$ is such that $P(y_1) = 0$. 

Just to avoid ambiguities, we stress that, henceforth in this proof, in the case $n = 0$: $\{x_1, \ldots, x_n, y_1\}:= \{y_1\}$, $\{x_1, \ldots, x_n, y_1, \ldots, y_k\}:= \{y_1, \ldots, y_k\} $, ~$a_1x_1 + \cdots + a_nx_n := 0$, $\{x_1, \ldots, x_n\} =\{1, \ldots, n\} := \emptyset$, and so on.
\medskip

\noindent{\bf Claim 1.} $P$ vanishes on ${\rm span}\{x_1, \ldots, x_n,y_1\}$.

\medskip

\noindent{\it Proof of Claim 1.} Given $a_1, \ldots , a_n, \beta \in \mathbb{K}$, set $a=a_1 x_1+\cdots + a_n x_n$. By the binomial formula \cite[Lemma 1.9]{dineen},
$$P(a+\beta y_1)=  \displaystyle \sum_{t=0}^m \binom{m}{t} {\check P} (a^{m-t},(\beta y_1)^t)=\sum_{t=0}^m \binom{m}{t} \beta ^t {\check P} (a^{m-t}, y_1^t).$$
It is enough to see that each of the terms of this sum is zero. Indeed:\\
\noindent $\bullet$ For  $t=0$,  ${\check P}(a^{m-t},y_1 ^t)={\check P}(a^m)=P(a)=0$ because $P$ vanishes on $W$.

\noindent $\bullet$ For $t=m$, ${\check P}(a^{m-t},y_1 ^t)={\check P}(y_1^m)= P(y_1)=0.$

\noindent $\bullet$ For $1 \leq t \leq m-1$,
\begin{align*}
 {\check P}(a^{m-t},y_1 ^t) &= {\check P}(a_1x_1+\cdots  +a_nx_n, \stackrel{(m-t)}{\dots} , a_1x_1+ \cdots + a_nx_n, y_1^t ) \\
 &= {\check P} \left( \sum_{i_1=1}^n a_{i_1} x_{i_1}, {\dots} , \sum_{i_{m-t}=1}^n a_{i_{m-t}}x_{i_{m-t}}, y_1^t\right) \\
% &= \sum _{i_1=1}^n \cdots \sum_{i_{m-t}=1}^n  {\check P}(a_{i_1}x_{i_1}, {\dots}, a_{i_{m-t}}x_{i_{m-t}}, y_1^t) \\
% &= \sum_{i_1,\dots, i_{m-t}=1}^n  {\check P}(a_{i_1}x_{i_1}, {\dots}, a_{i_{m-t}}x_{i_{m-t}}, y_1^t) \\
  &= \sum_{i_1,\dots, i_{m-t}=1}^n a_{i_1} \cdots a_{i_{m-t}}  {\check P}(x_{i_1}, {\dots},x_{i_{m-t}}, y_1^t) \\
  &= \sum_{i_1,\dots, i_{m-t}=1}^n a_{i_1} \cdots a_{i_{m-t}} P^{t,i_1,\dots , i_{m-t}}(y_1)=0,
\end{align*}
where we used that $\check P$ is $m$-linear, $y_1 \in Y_1$ and (\ref{ps1}). Claim 1 has been proved.

\medskip

\noindent{\bf Claim 2.} There is a sequence $(y_k)_{k=2}^\infty$ of vectors of $X$ and a sequence $(Y_k)_{k=2}^\infty$ of infinite dimensional subspaces of $Y$ such that, for every $k \geq 2$:\\
(i) $0 \neq y_k \in Y_k$, $y_{k-1} \notin Y_{k}\subseteq Y_{k-1}$ and $P(y_k)=0$.\\
(ii) Given $t \in \{1, \ldots, m-1\}$, $j\in \mathbb{N}$, $\alpha_1, \ldots, \alpha_{k-1} \in \mathbb{N}\cup \{0\}$ with $\alpha_1+\cdots + \alpha_{k-1}+j=m-t$, and given $i_1, \ldots, i_j \in \{1, \ldots, n\}$, the $t$-homogeneous polynomials
$$P_{y_1^{\alpha_1},\dots,y_{k-1}^{\alpha_{k-1}}}^{t, i_1,\dots, i_j}\colon X\longrightarrow \mathbb{K}~,~ x \in X \mapsto {\check P}(x_{i_1},\dots , x_{i_j}, y_1^{\alpha_1}, \ldots, y_{k-1}^{\alpha_{k-1}},x^t), $$
$$P_{y_1^{\alpha_1},\dots,y_{k-1}^{\alpha_{k-1}}}^{t} \colon X\longrightarrow \mathbb{K}~,~x \in X \mapsto{\check P}(y_1^{\alpha_1}, \dots, y_{k-1}^{\alpha_{k-1}},x^t), $$
vanish on $Y_k$.\\
(iii) $P$ vanishes on ${\rm span}\{x_1, \ldots, x_n, y_1, \ldots, y_k\}$.

\medskip

\noindent{\it Proof of Claim 2.} Let us show first how $y_2$ and $Y_2$ can be chosen according to the claim. Given $t \in \{1, \ldots, m-1\}$, $j\in \mathbb{N}$, $0 \leq \alpha_1 \leq m-t $ with $\alpha_1+j=m-t$, and $i_1, \ldots, i_j \in \{1, \ldots, n\}$, consider the $t$-homogeneous polynomials
$$P_{y_1^{\alpha_1}}^{t, i_1,\dots, i_{j}}\colon X\longrightarrow \mathbb{K}~,~x \in X\mapsto {\check P}(x_{i_1},\dots , x_{i_j}, y_1^{\alpha_1},x^t)$$
$$P_{y_1^{\alpha_1}}^{t}\colon X\longrightarrow \mathbb{K}~,~ x \in X \mapsto {\check P}(y_1^{\alpha_1},x^t), \mbox{ if } \alpha_1 > 0.$$
Note that if $\alpha_1 = 0$, then $P_{y_1^{\alpha_1}}^{t, i_1,\dots, i_{j}} = P^{t, i_1,\dots, i_{j}}$, which vanishes on $Y_1$ by $(\ref{ps1})$, hence it vanishes on any infinite dimensional subspace of $Y_1$ as well. Note also that we excluded the case $\alpha_1 = 0$ in the second polynomial because, in this case, we would have $t = m$, and we are confined to the cases $1 \leq t \leq m-1$. So, it is enough to handle the cases where $\alpha_1 \geq 1$. Since $P(x_{i_1}) = \cdots = P(x_{i_j}) = P(y_1) = 0$, in all these cases we have $P_{y_1^{\alpha_1},}^{t, i_1,\dots, i_{j}} \prec P$ and $P_{y_1^{\alpha_1},}^{t} \prec P$. By assumption, each of these polynomials vanishes on an infinite dimensional subspace of $Y_1$. Since $Y_1$ is infinite dimensional, by Lemma \ref{3} there is an infinite dimensional subspace $X_1$ of $Y_1$ such that, for all $t, \alpha_1, j$ and $i_1, \ldots, i_j$ as above,
\begin{equation}\label{ps2}
P_{y_1^{\alpha_1}}^{t, i_1,\dots, i_j}(x)= P_{y_1^{\alpha_1}}^{t}(x) =0 \mbox{ for every } x \in X_1.
\end{equation}
Let us consider two cases:\\
\noindent $\bullet$ $y_1 \in X_1$. In this case we choose $Y_2$ as any algebraic complement of ${\rm span}\{y_1\}$ in $X_1$. Of course, $Y_2$ is infinite dimensional subspace.

\noindent $\bullet$ $y_1 \notin X_1$. In this case we choose $Y_2 = X_1.$

In both cases, $Y_2$ is an infinite dimensional subspace of $X_1  \subseteq Y_1 \subseteq Y$, $y_1 \notin Y_2$ and, by (\ref{ps1}) and (\ref{ps2}), for all $t, \alpha_1, j$ and $i_1, \ldots, i_j$ as above,
\begin{equation}\label{p1}
P_{y_1^{\alpha_1}}^{t,i_1,\dots, i_{j}}(x)=P_{y_1^{\alpha_1}}^{t}(x)= P^{t,i_1,\dots, i_{j}}(x)=0 \mbox{ for every } x \in Y_2.
\end{equation}
Since $Y_2$ is an infinite dimensional subspace of $X$, by assumption there is $0 \neq y_2 \in Y_2$ such that $P(y_2) = 0$. To complete this step of the proof of Claim 2 we just have to show that $P$ vanishes on ${\rm span}\{x_1, \ldots, x_n,y_1, y_2\}$. To do so, let $a_1, \dots , a_n, \beta_1, \beta $ be given scalars and set  $b:=a_1 x_1+\cdots + a_n x_n+\beta_1 y_1$. The binomial formula gives
$$P(b+\beta y_2)=  \displaystyle \sum_{t=0}^m \binom{m}{t} {\check P}( b^{m-t},(\beta y_2)^t)=\sum_{t=0}^m \binom{m}{t} \beta ^t {\check P} (b^{m-t}, y_2^t).$$
It is enough to check that each term of this sum is zero. Indeed: \\
$\bullet$ For  $t=0$, ${\check P}(b^{m-t},y_1 ^t)={\check P}(b^m)=P(b)=0$ by Claim 1.

\noindent $\bullet$ For $t=m$, ${\check P}(b^{m-t},y_2 ^t)={\check P}(y_2^m)= P(y_2)=0.$

\noindent $\bullet$ For $1 \leq t \leq m-1$, consider the $(m-t)$-homogeneous polynomial
$$Q_1 \colon Y_2 \longrightarrow \mathbb{K}~,~Q_1(x)={\check P}(y_2^t,x^{m-t}). $$
Calling $a=a_1 x_1+\cdots + a_n x_n$, by the binomial formula,
\begin{align*}
{\check P}(b^{m-t},y_2^t) &= {\check P}((a+\beta_1y_1)^{m-t},y_2^t)=Q_1(a+\beta_1y_1)= \sum_{s=0}^{m-t} \binom{m-t}{s}  \stackrel{\vee}{Q_1} (a^{(m-t)-s},(\beta_1 y_1)^s) \\
&=\displaystyle \sum_{s=0}^{m-t} \binom{m-t}{s} {\check P} (a^{(m-t)-s},(\beta_1 y_1)^s,y_2^t) \\
%&= \sum_{s=0}^{m-t} \binom{m-t}{s} \beta_1 ^s {\check P}(a_1x_1+\cdots +a_nx_n, \stackrel{m-t-s}{\dots} , a_1x_1+ \cdots + a_nx_n, y_1^s,y_2^t ) \\
 &=\sum_{s=0}^{m-t} \binom{m-t}{s}\beta_1^s {\check P} \left( \sum_{i_1=1}^n a_{i_1} x_{i_1}, \dots , \sum_{i_{m-t-s}=1}^n a_{i_{m-t-s}}x_{i_{m-t-s}}, y_1^s, y_2^t\right) \\
% &= \sum_{s=0}^{m-t} \binom{m-t}{s} \beta_1^s \left( \sum _{i_1=1}^n \cdots \sum_{i_{m-t-s}=1}^n {\check P}(a_{i_1}x_{i_1},\dots, a_{i_{m-t-s}}x_{i_{m-t-s}}, y_1^s, y_2^t)\right) \\
% &= \sum_{s=0}^{m-t} \binom{m-t}{s} \beta_1^s \left( \sum_{i_1,\dots, i_{m-t-s}=1}^n {\check P}(a_{i_1}x_{i_1}, \dots, a_{i_{m-t-s}}x_{i_{m-t-s}}, y_1^s,y_2^t) \right) \\
  &= \sum_{s=0}^{m-t} \binom{m-t}{s} \beta_1^s \left( \sum_{i_1,\dots, i_{m-t-s}=1}^n a_{i_1} \cdots a_{i_{m-t-s}} {\check P}(x_{i_1},\dots,x_{i_{m-t-s}}, y_1^s,y_2^t) \right).
\end{align*}
All we have to do is to show that ${\check P}(x_{i_1}, \dots,x_{i_{m-t-s}}, y_1^s,y_2^t)=0$
for all $0 \leq s \leq m-t$ and $i_1, \ldots, i_{m-t-s} \in \{1, \ldots, n\}$ (recall that we are in the case $1\leq t \leq m-1$). Indeed,  using that $y_2 \in Y_2$ and $(\ref{p1})$ in the three cases below, we have:\\
$\bullet$ For $s = 0$,
$${\check P}(x_{i_1}, \dots,x_{i_{m-t-s}}, y_1^s,y_2^t)= {\check P}(x_{i_1},\dots,x_{i_{m-t}},y_2^t)= P^{t,i_1,\dots, i_{m-t}}(y_2) = 0.$$
$\bullet$ For $s = m-t$, we have  $m-t-s = 0$, hence,
$${\check P}(x_{i_1}, \dots,x_{i_{m-t-s}}, y_1^s,y_2^t)= {\check P}(y_1^{s},y_2^t)= P_{y_1^s}^{t}(y_2) = 0.$$
\noindent$\bullet$ For $ 1 \leq s \leq m-t-1$, we have $m-t-s > 0$, so
$${\check P}(x_{i_1}, \dots,x_{i_{m-t-s}}, y_1^s,y_2^t)= P_{y_1^s}^{t,i_1,\dots, i_{m-t-s}}(y_2) = 0.$$

We have just  established that  ${\check P}(b^{m-t},y_2^t)=0$ for every $1 \leq t \leq m-1.$ Therefore, $P(b+\beta y_2)=0$, which shows that $P$ vanishes on ${\rm span}\{x_1, \ldots, x_n, y_1, y_2\}$. This completes the selection of $y_2$ and $Y_2$.

Suppose now that $y_1, \ldots, y_k, Y_1, \ldots, Y_k$ have been chosen according to the claim. We shall show how $y_{k+1}$ and $Y_{k+1}$ can be chosen accordingly. Given $t \in \{1, \ldots, m-1\}$, $j \in \mathbb{N}$, $0 \leq \alpha_1,\dots , \alpha_{k-1} \leq m-t$, $1 \leq \alpha_k \leq m-t$ with $\alpha_1+\cdots + \alpha_k +j=m-t$, and $i_1, \ldots, i_j \in \{1, \ldots, n\}$, consider the $t$-homogeneous polynomials
$$P_{y_1^{\alpha_1}, \dots , y_k^{\alpha_k}}^{t, i_1, \dots , i_j} \colon X \longrightarrow \mathbb{K}~,~x \in X \mapsto {\check P}(x_{i_1}, \dots , x_{i_j}, y_1^{\alpha_1}, \dots ,  y_k^{\alpha_k},x^t), $$
$$P_{y_1^{\alpha_1}, \dots , y_k^{\alpha_k}}^t \colon X \longrightarrow \mathbb{K}~,~x \in X \mapsto {\check P}(y_1^{\alpha_1}, \dots ,y_k^{\alpha_k} ,x^t). $$
Note that we are not considering $\alpha_k =0 $ because, in this case, we have
$P_{y_1^{\alpha_1}, \dots , y_k^{\alpha_k}}^{t, i_1, \dots , i_j} = P_{y_1^{\alpha_1}, \dots , y_{k-1}^{\alpha_{k-1}}}^{t, i_1, \dots , i_j} $ and $P_{y_1^{\alpha_1}, \dots , y_k^{\alpha_k}}^t = P_{y_1^{\alpha_1}, \dots , y_{k-1}^{\alpha_{k-1}}}^t$, polynomials which vanish on $Y_{k}$ by the induction hypothesis, hence they vanish on any infinite dimensional subspace of $Y_{k}$ as well. %Note also that we excluded the case $\alpha_1 = \cdots =  0$ in the second polynomial because, in this case, we would have $t = m$, and we are confined to the cases $1 \leq t \leq m-1$. So, it is enough to handle the cases where $\alpha_1 \geq 1$.
Since $P(x_{i_1}) = \cdots = P(x_{i_j}) = P(y_1) = \cdots = P(y_k) =0$, in all these cases we have $P_{y_1^{\alpha_1}, \dots , y_k^{\alpha_k}}^{t, i_1, \dots , i_j} \prec P$ and $P_{y_1^{\alpha_1}, \dots , y_k^{\alpha_k}}^t \prec P$. Since $Y_k$ is an infinite dimensional subspace of $X$, by assumption each of these polynomials vanishes on an infinite dimensional subspace of $Y_k$. By Lemma \ref{3} there is an infinite dimensional subspace $X_k$ of $Y_k$ such that, for all $t, \alpha_1, \ldots, \alpha_k, j$ and $i_1, \ldots, i_j$ as above,
\begin{equation}\label{lk9u} P_{y_1^{\alpha_1}, \dots,  y_k^{\alpha_k}}^{t,i_1,\dots, i_{j}}(x)=P_{y_1^{\alpha_1}, \dots , y_k^{\alpha_k}}^{t}(x)=0 \mbox{ for every } x \in X_k.
\end{equation}
Let us consider two cases:\\
$\bullet$ $y_k \in X_k$. In this case we choose $Y_{k+1}$ as any algebraic complement of ${\rm span}\{y_k\}$ in $X_k$. Of course, $Y_{k+1}$ is infinite dimensional subspace.

\noindent $\bullet$ $y_k \notin X_k$. In this case we choose $Y_{k+1} = X_k.$

In both cases, $Y_{k+1}$ is an infinite dimensional subspace of $X_k  \subseteq Y_k \subseteq Y_1 \subseteq Y$, $y_k \notin Y_{k+1}$ and, by (\ref{ps1}) and (\ref{lk9u}), for all $t, \alpha_1,\ldots, \alpha_k, j$ and $i_1, \ldots, i_j$ as above,
\begin{equation}\label{phi}
P_{y_1^{\alpha_1}, \dots,  y_k^{\alpha_k}}^{t,i_1,\dots, i_{j}}(x)=P_{y_1^{\alpha_1}, \dots , y_k^{\alpha_k}}^{t}(x)= P^{t,i_1,\dots, i_{j}}(x)=0 \mbox{ for every } x \in Y_{k+1}.
\end{equation}
As mentioned before, the above also holds for $\alpha_k = 0$, by the induction hypothesis, because $Y_{k+1} \subseteq Y_k$.

Since $Y_{k+1}$ is an infinite dimensional subspace of $X$, by assumption there is $0 \neq y_{k+1} \in Y_{k+1}$ such that $P(y_{k+1}) = 0$. To complete this step of the proof of Claim 2 we just have to show that $P$ vanishes on ${\rm span}\{x_1, \ldots, x_n,y_1,\ldots,y_k, y_{k+1}\}$. To do so, let $a_1, \dots , a_n, \beta_1, \ldots, \beta_k, \beta $ be given scalars and set  $c:=a_1 x_1+\cdots + a_n x_n+\beta_1 y_1+ \cdots + \beta_k y_k$. The binomial formula gives
$$P(c+\beta y_{k+1})=  \displaystyle \sum_{t=0}^m \binom{m}{t} {\check P} (c^{m-t},(\beta y_{k+1})^t)=\sum_{t=0}^m \binom{m}{t} \beta ^t {\check P} (c^{m-t}, y_{k+1}^t).$$
It is enough to check that each term of this sum is zero. Indeed: \\
$\bullet$ For  $t=0$, since $c \in {\rm span}\{x_1, \ldots, x_n,y_1,\ldots,y_k\}$ and $P$ vanishes on this subspace by the inductions hypothesis, we have ${\check P}(c^{m-t},y_{k+1} ^t)={\check P}(c^m)=P(c)=0$.\\
$\bullet$ For $t=m$, ${\check P}(c^{m-t},y_{k+1} ^t)={\check P}(y_{k+1}^m)= P(y_{k+1})=0.$\\
$\bullet$ For $1 \leq t \leq m-1$, consider the $(m-t)$-homogeneous polynomial
$$Q_k \colon Y_{k+1} \longrightarrow \mathbb{K} ~,~Q_k(x)={\check P}(y_{k+1}^t,x^{m-t}). $$
Putting $a=a_1 x_1+\cdots + a_n x_n$ and $b=\beta _1 y_1+ \dots +\beta_k y_k$, the binomial formula gives
{\footnotesize
\begin{align*}
{\check P}&(c^{m-t},y_{k+1}^t) = {\check P}((a+b)^{m-t},y_{k+1}^t)=Q_k(a+b)= \sum_{s=0}^{m-t} \binom{m-t}{s} \stackrel{\vee}{Q_k}(a^{(m-t)-s},b^s)\\
&=\displaystyle \sum_{s=0}^{m-t} \binom{m-t}{s} {\check P} (a^{(m-t)-s},b^s,y_{k+1}^t) \\
%&= \sum_{s=0}^{m-t} \binom{m-t}{s} A((a_1x_1+\cdots +a_nx_n)^{m-t-s}, (\beta_1 y_1+ \cdots \beta_k y_k)^s, y_{k+1}^t ) \\
 &=\sum_{s=0}^{m-t} \binom{m-t}{s} {\check P} \left( \sum_{i_1=1}^n a_{i_1} x_{i_1}, \dots , \sum_{i_{m-t-s}=1}^n a_{i_{m-t-s}}x_{i_{m-t-s}},  \sum_{r_{1}=1}^k \beta_{r_1}y_{r_{1}}, \dots, \sum_{r_{s}=1}^k \beta_{r_{s}}y_{r_{s}}, y_{k+1}^t\right) \\
% &= \sum_{s=0}^{m-t} \binom{m-t}{s}  \left( \sum _{i_1=1}^n \cdots \sum_{i_{m-t-s}=1}^n \sum_{r_{1}=1}^2 \dots \sum_{r_{s}=1}^2 A(a_{i_1}x_{i_1}, \stackrel{m-t-s}{\dots}, a_{i_{m-t-s}}x_{i_{m-t-s}}, y_1^s, y_2^t)\right) \\
 &= \sum_{s=0}^{m-t} \binom{m-t}{s}  \left( \sum_{i_1,\dots, i_{m-t-s}=1}^n \sum_{r_1,\dots, r_{s}=1}^k {\check P}(a_{i_1}x_{i_1}, \dots, a_{i_{m-t-s}}x_{i_{m-t-s}}, \beta_{r_1}y_{r_{1}}, \dots, \beta_{r_{s}}y_{r_{s}} ,y_{k+1}^t) \right) \\
  &= \sum_{s=0}^{m-t} \binom{m-t}{s}  \left( \sum_{i_1,\dots, i_{m-t-s}=1}^n \sum_{r_1,\dots, r_{s}=1}^k a_{i_1} \cdots a_{i_{m-t-s}} \beta_{r_1}\cdots \beta_{r_s} {\check P}(x_{i_1}, \dots,x_{i_{m-t-s}}, y_{r_{1}}, \dots, y_{r_{s}} ,y_{k+1}^t) \right).
\end{align*} }
All we have to do is to show that ${\check P}(x_{i_1}, \dots,x_{i_{m-t-s}},y_{r_{1}}, \dots, y_{r_{s}} ,y_{k+1}^t)=0$ for all
 $0 \leq s \leq m-t, 1 \leq i_1, \ldots, i_{m-t-s}\leq n, 1 \leq r_1, \ldots, r_s \leq k$ (remember that we are in the case $1\leq t \leq m-1$). Indeed, using that $y_{k+1} \in Y_{k+1}$ and (\ref{phi}) in the three cases below, we have:\\
 $\bullet$ For $s = 0$, $m-t-s=m-t >0$ because $t\leq m-1$, hence
 $${\check P}(x_{i_1}, \dots , x_{i_{m-t-s}}, y_{r_1}, \dots , y_{r_s}, y_{k+1}^t)={\check P}(x_{i_1}, \dots , x_{i_{m-t}}, y_{k+1}^t)=P^{t,i_1,\dots, i_{m-t}}(y_{k+1})\stackrel{{\rm(\ref{phi})}}{=}0.$$
 $\bullet$ For $s = m-t$, $s \neq 0$ and $m-t-s=0$. Since $1 \leq r_1, \ldots, r_s \leq k$, for every $i \in \{1,\ldots, s\}$ there is $j \in \{1, \ldots, k\}$ such that $y_{r_i} = y_j$. Thus, there are $\alpha_1, \ldots \alpha_k \in \mathbb{N} \cup \{0\}$ such that $\alpha_1 +  \cdots +  \alpha_k = m -t =s$ and
 $${\check P}(y_{r_1}, \dots , y_{r_s},y_{k+1}^t)={\check P}(y_1^{\alpha_1}, \dots ,y_k^{\alpha_k},y_{k+1}^t). $$
 Therefore,
 \begin{align*} {\check P}(x_{i_1}, \dots , x_{i_{m-t-s}}, y_{r_1}, \dots , y_{r_s}, y_{k+1}^t)&={\check P}(y_{r_1}, \dots , y_{r_s}, y_{k+1}^t)\\
 & = {\check P}(y_1^{\alpha_1}, \dots ,y_k^{\alpha_k},y_{k+1}^t)\\
 & = P^t_{y_1^{\alpha_1}, \dots,  y_k^{\alpha_k}}(y_{k+1}) \stackrel{{\rm(\ref{phi})}}{=} 0.
 \end{align*}
$\bullet$ For $1 \leq s \leq m-t-1$, we have $m-t-s >0$, so, as before, there are $\alpha_1, \ldots \alpha_k \in \mathbb{N} \cup \{0\}$ such that $\alpha_1 +  \cdots +  \alpha_k = s$ and
 $${\check P}(x_{i_1}, \dots , x_{i_{m-t-s}}, y_{r_1}, \dots , y_{r_s}, y_{k+1}^t)= {\check P}(x_{i_1}, \dots , x_{i_{m-t-s}}, y_1^{\alpha_1}, \dots , y_k^{\alpha_k}, y_{k+1}^t).$$
 Note that, since $s \geq 1$, $\alpha_1, \ldots, \alpha_k$ are not all zero. As always, let $j$ be such that $\alpha_1 + \cdots + \alpha_k + j = m-t$. In this case, $j = m-t-s$. Therefore,
\begin{align*}{\check P}(x_{i_1}, \dots , x_{i_{m-t-s}}, y_{r_1}, \dots , y_{r_s}, y_{k+1}^t)&= {\check P}(x_{i_1}, \dots , x_{i_{j}}, y_1^{\alpha_1}, \dots , y_k^{\alpha_k}, y_{k+1}^t)= P_{y_1^{\alpha_1},\dots ,y_k^{\alpha_k}}^{t,i_1, \dots, i_{j}}(y_{k+1})\stackrel{{\rm(\ref{phi})}}{=}0,
\end{align*}
which proves that $P$ vanishes on ${\rm span}\{x_1, \ldots, x_n,y_1,\ldots,y_k, y_{k+1}\}$. This completes the proof of Claim 2.

\medskip

\noindent{\bf Claim 3.} The set $\{x_1, \ldots, x_n\}\cup\{y_k : k \in \mathbb{N}\}$ is linearly independent.

\medskip

\noindent{\it Proof of Claim 3.} It is enough to prove that $\{x_1, \dots, x_n, y_1, \dots y_k\}$ is linearly independent for every $k \in \mathbb{N}$. To do so, let $\alpha_1, \dots , \alpha_n, \beta_1, \dots , \beta_k$ be scalars such that $\alpha_1 x_1+ \dots + \alpha_n x_n + \beta_1y_1+ \dots + \beta_k y_k=0$. We know that $y_i \in Y_i \subseteq Y$ for every $i \in \mathbb{N}$, hence
$$W \ni \alpha_1 x_1+ \dots + \alpha_n x_n = -\beta_1y_1 - \dots  -\beta_k y_k \in Y.$$
Since $X=W \oplus Y$, it follows that
$$\alpha_1 x_1+ \dots + \alpha_n x_n =0 =  \beta_1y_1 + \dots  +\beta_k y_k.
$$
We have $\alpha_1 = \cdots = \alpha_n=0$ because $\{x_1, \dots, x_n\}$ is a basis for $W$. Suppose that  $\beta_1 \neq 0$.
%$$y_1= -\frac{\beta_2 y_2}{\beta_1}-\cdots - \frac{\beta_{k}y_{k}}{\beta_1}.$$
%Como $Y_2$ é um subespaço vetorial,
By Claim 2 we have $Y_k \subseteq \cdots \subseteq Y_3 \subseteq Y_2$ and
$y_i \in Y_i$ for each $i \in \{1,\dots , k\}$. Therefore,
$$y_1= -\frac{\beta_2 y_2}{\beta_1}-\cdots - \frac{\beta_{k}y_{k}}{\beta_1} \in Y_2.$$
This is a contradiction because $y_1 \notin Y_2$ by Claim 2. This proves that $\beta_1=0$. It follows that $\beta_2y_2 + \dots  +\beta_k y_k = 0$. A repetition of latter reasoning gives $\beta_2 = 0$. After finitely many repetitions of this process we obtain $\beta_1 = \cdots = \beta_k = 0$, proving Claim 3.

\medskip

Claim 3 ensures that ${\rm span}\{x_1, \ldots, x_n, y_1, y_2, \ldots\}$ is an infinite dimensional subspace of $X$, of course it contains $W$, and $P$ vanishes on it by Claim 2. Therefore, the zero set of $P$ is finitely lineable.~~~~~~~~~~~~~~~~~~~~~~~~~~~~~~~~~~~~~~~~~~~~~~~~~~~~~~~~~~~~~~~~~~~~~~~~~~~~~~~~~~~~~~~~~~~~~~~~~~~~~~~~~~~~~~~~~~~~~~~~$\square$

\section{Applications}

We start with consequences of the main result that hold in the complex and real cases. The first one is a refined version of Lemma \ref{3} for homogeneous polynomials.%, which will be helpful later.

\begin{corollary} Let $P_1, \ldots, P_k$ be homogeneous polynomials on $X$. If, for every infinite dimensional subspace $Y$ of $X$, each $P_i$ vanishes on a nonzero vector of $Y$ and each homogeneous polynomial $Q \prec P_i, i = 1, \ldots, k,$ vanishes on an infinite dimensional subspace of $Y$, then the set $\bigcap\limits_{i=1}^k P_i^{-1}(0)$ is finitely lineable. %In particular, the zero set of the (non-necessarily homogeneous) polynomial $P_1 + \cdots + P_k$ is finitely lineable.
\end{corollary}

\begin{proof} Let $W$ be finite dimensional subspace of $X$ contained in $\bigcap\limits_{i=1}^k P_i^{-1}(0)$. As $P_1$ fulfills the assumptions of Theorem \ref{main} and $W \subseteq P_1^{-1}(0)$, there is an infinite dimensional $X_1$ of $X$ such that $W \subseteq X_1 \subseteq P_1^{-1}(0)$. The restriction $P_2|_{X_1}$ of $P_2$ to $X_1$ is a homogeneous polynomial on an infinite dimensional space which, by assumption,  fulfills the conditions Theorem \ref{main}. As $W \subseteq P_2^{-1}(0)$, there is  an infinite dimensional subspace $X_2$ of $X_1$ such that $W \subseteq X_2 \subseteq  P_2^{-1}(0)$. As $X_2 \subseteq X_1$, $P_1$ vanishes on $X_2$, hence $W \subseteq X_2 \subseteq  P_1^{-1}(0) \cap P_2^{-1}(0)$. Just repeat the procedure finitely many times to get the result.
\end{proof}

A particular case of Theorem \ref{main} gives a contribution to the subject of pointwise lineability, introduced in \cite{Pellegrino} and developed in, e.g., \cite{mikaela, calderon, anselmogeivison}:

\begin{corollary} Suppose that a homogeneous polynomial $P$ on $X$ satisfying the assumptions of Theorem \ref{main} vanishes on a point $x \in X$. Then there is an infinite dimensional subspace of $X$ containing $x$ and contained in the zero set of $P$.
\end{corollary}

\begin{proof} Just apply Theorem \ref{main} to $W = {\rm span}\{x\}$.
\end{proof}

 Given a cardinal number $\kappa$, we say that a subset $A$ of a topological vector space $X$ is {\it finitely $\kappa$-spaceable} if, for every finite dimensional $W$ of $X$ such that $W \subseteq A \cup \{0\}$, there is a closed $\kappa$-dimensional subspace $V$ of $X$ such that $W \subseteq V \subseteq A \cup \{0\}$. In the language of \cite{Favaro}, the zero set of homogeneous polynomial on $X$ is {finitely $\kappa$-spaceable} if and only if it is $(n, \kappa)$-spaceable for every $n \in \mathbb{N}$.

\begin{corollary}\label{qro9} Let $P$ be a continuous homogeneous polynomial on an infinite dimensional topological vector space $X$ satisfying the assumptions of Theorem \ref{main}. Then, the zero set of $P$ is finitely $\aleph_0$-spaceable. If, in addition, $X$ is complete metrizable, in particular if $X$ is a Banach space, then the zero set of $P$ is finitely $\mathfrak{c}$-spaceable.
\end{corollary}

\begin{proof} The first statement follows from Theorem \ref{main} because the zero set of a continuous homogeneous polynomial is a closed set. As the dimension of any complete metrizable infinite dimensional topological vector space is at least $\mathfrak{c}$ (see \cite{popoola}), the second statement follows from the first.
\end{proof}

In the complex case, the following corollary is an extension of the Plichko-Zagorodnyuk theorem \cite{plichko} mentioned in the Introduction, the one that motivated our research.

\begin{corollary}\label{ghnv} The zero set of any homogeneous polynomial on an infinite dimensional complex linear space is finitely lineable.
\end{corollary}

\begin{proof} The Plichko-Zagorodnyuk theorem assures that any homogeneous polynomial on an infinite dimensional complex linear space fulfills the assumptions of Theorem \ref{main}, therefore the result follows.
\end{proof}

Again, a particular case gives a contribution to pointwise lineability:

\begin{corollary}\label{l3dn} Suppose that a homogeneous polynomial $P$ on a complex infinite dimensional  linear space  $X$ vanishes on a point $x \in X$. Then there is an infinite dimensional subspace of $X$ containing $x$ and contained in the zero set of $P$.
\end{corollary}

Combining the Plichko-Zagorodnyuk theorem with Corollary \ref{qro9} we get the following:

\begin{corollary}\label{l4dn} The zero set of any continuous homogeneous polynomial $P$ on an infinite dimensional complex topological vector space $X$ is finitely $\aleph_0$-spaceable. If, in addition, $X$ is complete and metrizable, in particular if $X$ is a Banach space, then the zero set of $P$ is finitely $\mathfrak{c}$-spaceable.
\end{corollary}

Now we draw our attention to the real case. In \cite[Remark 1]{realstory}, the authors proved that every separable real Banach space supports a positive definite 2-homogeneous polynomial. Since every infinite dimensional Banach space admits a separable closed infinite dimensional subspace -- in strong contrast to the complex case -- the assumptions of Theorem \ref{main} are not fulfilled by all homogeneous polynomials on any infinite dimensional real Banach space. Of course, this shows that general results of the type of Corollaries \ref{ghnv}, \ref{l3dn} and \ref{l3dn}  are not true in the real case. In order to show that our results are useful in the real case, we shall finish the paper giving examples of specific polynomials on real spaces for which our results apply.

An $m$-homogeneous polynomial $P$ on a linear space $X$ is {\it of finite type} if there are $k \in \mathbb{N}$ and linear functionals $\varphi_{1,1}, \ldots, \varphi_{1,m}, \ldots, \varphi_{k,1},\ldots, \varphi_{k,m}$ on $X$ such that
$$P(x) = \sum_{j=1}^k \varphi_{j,1}(x) \cdots \varphi_{j,m}(x) \mbox{ for every } x \in X. $$

The fact that each homogeneous polynomial $P$ of finite type vanishes on an infinite dimensional space follows immediately from Lemma \ref{3}; but the finite lineability of $P^{-1}(0)$ does not. In view of Corollary \ref{ghnv}, the next result is stated only for real scalars.

\begin{proposition}\label{umd3} The zero set of any homogeneous polynomial of finite type on any infinite dimensional real linear space is finitely lineable.
\end{proposition}

\begin{proof} Let $P$ be an $m$-homogeneous polynomial of finite type on an infinite dimensional linear space $X$. It is well known that $P$ can be written as
$$P(x) = \sum_{j=1}^k a_j \varphi_j(x)^m \mbox{ for every } x \in X, $$
where $k \in \mathbb{N}$, $a_1, \ldots,a_k$ are scalars and $\varphi_1, \ldots, \varphi_k$ are linear functionals on $X$ (see \cite[p.\,42]{dineen}). Let $Y$ be an arbitrary infinite dimensional linear subspace of $X$. As $Y$ is infinite dimensional, for every $j = 1, \ldots, k$, the kernel of the restriction of $\varphi_j$ to $Y$, denoted by $\ker(\varphi_j|_Y)$, is an infinite dimensional subspace of $Y$ on which $\varphi_j$ vanishes. By Lemma \ref{3} there is an infinite dimensional subspace $Z$ of $Y$ contained in $\bigcap\limits_{j=1}^k \ker(\varphi_j|_Y)$. Thus, $P$ vanishes on $Z$, in particular, $P$ vanishes on some nonzero vector of $Y$. For every $(x_1, \ldots, x_m) \in X^m$,
$${\check P}(x_1, \ldots, x_m) = \sum_{j=1}^k a_j\varphi_j(x_1) \cdots \varphi_j(x_m). $$
Let $1 \leq t \leq m-1$ and let $Q \prec P$ be a $t$-homogeneous polynomial. Then there are $x_1,\dots, x_n \in X$ on which $P$ vanishes, and $\alpha_1, \dots, \alpha_n \in \mathbb{N}\cup\{0\}$ with $\alpha_1+ \cdots + \alpha_n+t=m$ such that $Q(x)={\check P}(x_1^{\alpha_1}, \dots , x_n^{\alpha_n}, x^t)$ for every $x \in X$. It is plain that there are $y_1, \ldots, y_{m-t} \in \{x_1, \ldots, x_n\}$ such that, for every $x \in X$,
$$Q(x)={\check P}(x_1^{\alpha_1}, \dots , x_n^{\alpha_n}, x^t) = {\check P}(y_1, \ldots, y_{m-t}, x^t) = \sum_{j=1}^k a_j\varphi_j(y_1) \cdots \varphi_j(y_{m-t}) \varphi_j(x)^{t}.$$
Since $Z \subseteq \bigcap\limits_{j=1}^k \ker(\varphi_j|_Y)$ and $t \geq 1$, we have that $Q$ vanishes on the infinite dimensional subspace $Z$ of $Y$. We have proved that $P$ fulfills the assumptions of Theorem \ref{main}, therefore the zero set of $P$ is finitely lineable.
%\bigskip
%\bigskip
%For $j = 1, \ldots, k$, consider the $m$-homogeneous polynomial $x \in X \mapsto P_j(x) := \varphi_j(x)^m$. For every $(x_1, \ldots, x_m) \in X^m$,%$$$(P_j)^\vee(x_1, \ldots, x_m) = \varphi_j(x_1) \cdots \varphi_j(x_m). $$
%Let $1 \leq t \leq m-1$ and let $Q \prec P_j$ be a $t$-homogeneous polynomial. Then there are $x_1,\dots, x_n \in X$ on which $P_j$ vanishes, and $\alpha_1, \dots, \alpha_n \in \mathbb{N}\cup\{0\}$ with $\alpha_1+ \cdots + \alpha_n+t=m$ such that $Q(x)=(P_j)^\vee(x_1^{\alpha_1}, \dots , x_n^{\alpha_n}, x^t)$ for every $x \in X$. It is plain that there are $y_1, \ldots, y_{m-t} \in \{x_1, \ldots, x_n\}$ such that
%Dado $m \in \mathbb{N}$. Sejam $X$ um espaço vetorial, $P \in \mathcal{P}(^mX)$ e $Y$ um subespaço de $X$. Para cada  $1 \leq t \leq m-1$, denotaremos por $Q \prec P$ o polinômio $t$-homogêneo $Q: Y \to \mathbb{K},$ definido por
%\begin{align*}Q(x)&=(P_j)^\vee(x_1^{\alpha_1}, \dots , x_n^{\alpha_n}, x^t) = (P_j)^\vee(y_1, \ldots, y_{m-t}, x^t) = \varphi_j(y_1) \cdots \varphi_j(y_{m-t}) \varphi_j(x)^{t}
%\end{align*}
% for every $ x \in X.$ Since $t \leq m-1$, the number $\varphi_j(y_1)$ certainly appears  on the expression above. But $y_1 = x_i$ for some $i \in \{1, \ldots, n\}$, so $\varphi_j(y_1)^m = P_j(y_1) = P_j(x_i) = 0$, that is, $\varphi_j(y_1) =0$; hence $Q(x) = 0$ for every $x \in x$. This proves that the only polynomial $Q \prec P_j$ is the null polynomial. Therefore, the polynomials $P_1, \ldots, P_k$ fulfill the assumptions of Corollary
\end{proof}

\begin{remark}\label{ytnb}\rm Still in the real case:\\
(a) On the one hand, Theorem \ref{main} may be applied to specific homogeneous polynomials that are not of finite type. On the other hand, as to arbitrary polynomials belonging to subspaces larger than the finite type polynomials, one cannot go much further: It is easy to see that the positive definite 2-homogeneous polynomial $(a_j)_{j=1}^\infty \in \ell_\infty \mapsto \sum\limits_{j=1}^\infty \frac{a_j^2}{2^j} \in \mathbb{R}$ belongs to the closure, with respect to the usual norm of spaces of continuous homogeneous polynomials, of the space of finite type polynomials. Actually, this polynomial is nuclear. \\
(b) A result similar to Corollary \ref{l3dn}  for  polynomials of finite type on infinite dimensional real linear spaces, and a result similar to Corollary \ref{l4dn} for continuous  polynomials of finite type on infinite dimensional real topological vector spaces follow Proposition \ref{umd3}.

\end{remark}

%Still in the real case, as to classes of polynomials larger than the finite type polynomials, the next example shows that, in general, one cannot go much further.
%
%\begin{example}\rm Consider the continuous 2-homogeneous polynomial
%$$P \colon \ell_\infty \longrightarrow \mathbb{K}~,~P\left((a_j)_{j=1}^\infty\right) = \sum_{j=1}^\infty \frac{a_j^2}{2^j}. $$
%It is easy to check that $P$ is the limit, in the usual norm of spaces of continuous homogeneous polynomials, of the sequence $(P_n)_{n=1}^\infty$ of continuous 2-homogeneous polynomials of finite type defined by $P_n\left((a_j)_{j=1}^\infty\right) = \sum\limits_{j=1}^n \frac{a_j^2}{2^j}$. Therefore, $P$ belongs to the closure of the space of finite type polynomials, but, in the real case, $P$ is positive definite.
%\end{example}

%$$P_n\left((a_j)_{j=1}^\infty\right) = \sum\limits_{j=1}^n \frac{a_j^2}{2^j} = \sum\limits_{j=1}^n \frac{\varphi_j(a)^2}{2^j}$$
%$$\varphi_k \colon \ell_\infty \longrightarrow \mathbb{R}~,~\varphi_k(a) = a_k. $$

%bigskip
%
%
%Next corollary is less general than the theorem, but it will be helpful in the subsequent application.
%
%\begin{corollary} Suppose that, for every infinite dimensional subspace $Y$ of $X$, each $t$-homoge-neous polynomial on $X$, $t = 1,\ldots,m$, vanishes on an infinite dimensional subspace of $Y$. Then the zero set of any $m$-homogeneous polynomial on $X$ is finitely lineable.
%\end{corollary}
%
%Next corollary is refined version of Lemma \ref{3} for homogeneous polynomials.

\begin{remark}\rm Let us see that the multilinear counterpart of the polynomial problem handled in this paper is not difficult. To do so, let $X_1, \ldots, X_m$ be (real or complex) linear spaces and let $A \colon X_1 \times \cdots \times X_m \longrightarrow \mathbb{K}$ be an $m$-linear form. It is obvious that, for every $ j = 1, \ldots, m$, $A$ vanishes on a $(\dim X_j)$-dimensional space. Furthermore,  using Lemma \ref{3}, one can check directly (not using induction) that, if $X_j$ is infinite dimensional for some $j \in \{1, \ldots, m\}$, then the zero set of $A$ is finitely lineable.
\end{remark}
%\bigskip
\noindent{\bf Acknowledgement.} The authors thank Thiago R. Alves for his helpful suggestions.

\bigskip
\noindent Mikaela Aires~~~~~~~~~~~~~~~~~~~~~~~~~~~~~~~~~~~~~~~~~~~~~~Geraldo Botelho~\\
Instituto de Matem\'atica e Estat\'istica~~~~~~~~~~~~~~~Instituto de  Matem\'atica e Estat\'istica\\
Universidade de S\~ao Paulo~~~~~~~~~~~~~~~~~~~~~~~~~~~~~\hspace*{0,1em}Universidade Federal de Uberl\^andia\\
05.508-090 -- S\~ao Paulo -- Brazil~~~~~~~~~~~~~~~~~~~~~~\,\hspace*{0,1em}38.400-902 -- Uberl\^andia -- Brazil\\
e-mail: mikaela\_aires@ime.usp.br~~~~~~~~~~~~~~~~~~~~~\,e-mail: botelho@ufu.br
\bigskip

%\noindent Geraldo Botelho~~~~~~~~~~~~~~~~~~~~~~~~~~~~~~~~~~~~~~Mikaela Aires\\
%Faculdade de Matem\'atica~~~~~~~~~~~~~~~~~~~~~~~~~~Instituto de Matem\'atica e Estat\'istica\\
%Universidade Federal de Uberl\^andia~~~~~~~~~~~~~Universidade de S\~ao Paulo\\
%38.400-902 -- Uberl\^andia -- Brazil~~~~~~~~~~~~~~~~~05.508-090 -- S\~ao Paulo -- Brazil\\
%e-mail: botelho@ufu.br~~~~~~~~~~~~~~~~~~~~~~~~~\,~~~~~e-mail: mikaela\_aires@ime.usp.br

%\medskip
%
%\noindent Ariel S. Santiago\\
%Departamento de Matemática\\
%Universidade Federal de Minas Gerais\\
%???? -- Belo Horizonte -- Brazil\\
%e-mail: ?????

%\bigskip
%\bigskip
%
%$F$ preserva multimorfismos de Riesz se as extensões de Aron-Berner de $F$-valued MR são MR
%
%\bigskip
%
%Teorema. Se $F_n$ preserva MR para todo $n$ então $(\oplus_n F_n)_p$ e $(\oplus_n F_n)_0$ preservam MR

\end{document}